\documentclass[11pt]{amsart}
\usepackage{a4wide}
\usepackage{kotex}
\usepackage{amsmath}
\usepackage{amsfonts}
\usepackage{amscd}

\usepackage{amssymb}
\usepackage{graphicx}
\usepackage{dsfont}
\usepackage{color}
\usepackage{bbm}
\usepackage[colorlinks]{hyperref}
\hypersetup{citecolor=blue,}
\hypersetup{urlcolor=black,}
\usepackage{url}

\usepackage{breakurl}

\usepackage{sseq}
\usepackage{tikz-cd}
\usepackage[abs]{overpic} 
\usepackage{bm}
\usepackage{caption}
\usepackage{subcaption}

\theoremstyle{theorem}
\newtheorem{TheoremA}{Theorem}

\newtheorem{theorem}{Theorem}[section]
\newtheorem{lemma}[theorem]{Lemma}

\newtheorem{proposition}[theorem]{Proposition}
\newtheorem{corollary}[theorem]{Corollary}

\newtheorem{question}[theorem]{Question}
\newtheorem*{question*}{Question}

\theoremstyle{definition}
\newtheorem*{definition*}{Definition}
\newtheorem{definition}[theorem]{Definition}

\newtheorem*{example*}{Example}
\newtheorem{example}[theorem]{Example}
\newtheorem*{observation*}{Observation}

\newtheorem*{Goal*}{Goal}

\newtheorem*{Assumption*}{Assumption}

\theoremstyle{remark}
\newtheorem*{remark*}{Remark}
\newtheorem{remark}[theorem]{Remark}

\numberwithin{equation}{section}

\newcommand{\ow}{\omega}

\newcommand{\lda}{\lambda}

\newcommand{\p}{\partial}

\newcommand{\C}{{\mathbb{C}}}
\newcommand{\R}{{\mathbb{R}}}
\newcommand{\Q}{{\mathbb{Q}}}
\newcommand{\Z}{{\mathbb{Z}}}
\newcommand{\N}{{\mathbb{N}}}

\newcommand{\F}{{\mathbb{F}}}

\DeclareMathOperator{\id}{id}

\DeclareMathOperator{\rk}{rk}

\DeclareMathOperator{\Spec}{Spec}

\DeclareMathOperator{\CZ}{CZ}
\DeclareMathOperator{\RS}{RS}

\DeclareMathOperator{\SH}{SH}
\DeclareMathOperator{\HF}{HF}
\DeclareMathOperator{\CF}{CF}

\DeclareMathOperator{\Ho}{H}
\DeclareMathOperator{\st}{st}

\DeclareMathOperator{\GH}{GH}

\begin{document}

\title{On dynamically convex contact manifolds and filtered symplectic homology}

\author{Myeonggi Kwon}
\address{Department of Mathematics Education, and Institute of Pure and Applied Mathematics, Jeonbuk National University, Jeonju 54896, Republic of Korea}
\email{mkwon@jbnu.ac.kr}

\author{Takahiro Oba}
\address{Department of Mathematics, Osaka University, Toyonaka, Osaka 560-0043, Japan}
\email{taka.oba@math.sci.osaka-u.ac.jp}



\begin{abstract}
In this paper we are interested in characterizing the standard contact sphere in terms of dynamically convex contact manifolds which admit a Liouville filling with vanishing symplectic homology. We first observe that if the filling is flexible, then those contact manifolds are contactomorphic to the standard contact sphere. We then investigate quantitative geometry of those contact manifolds focusing on similarities with the standard contact sphere in filtered symplectic homology.


\end{abstract}

\maketitle


\section{Introduction}

Let $(W, \lda)$ be a Liouville domain, meaning that $(W, d\lda)$ is an exact symplectic manifold with boundary and the Liouville vector field $X$, determined by $\iota_{X} d\lda  = \lda$, is pointing outward along the boundary $\p W$. The Liouville form $\lda$ induces a contact form $\alpha := \lda|_{\p W}$ on $\p W$. In this case, we say $(W, \lda)$ is a \emph{Liouville filling} of $(\p W, \xi)$. Throughout the paper, we will assume that $(W, \lda)$ is topologically simple, i.e. the inclusion $\p W \rightarrow W$ is $\pi_1$-injective and $c_1(TW) = 0$; see Remark \ref{rem: globalassumption}. 
A typical example is the \emph{standard symplectic ball} $
B^{2n} : = \{\mathbf{x}+ i\mathbf{y} \in \C^n \;|\; ||\mathbf{x}||^2 + ||\mathbf{y}||^2 \leq 1\}$
with the Liouville form $
\lda_{\st} : = \frac{1}{2}(\mathbf{x} \text{d} \mathbf{y} - \mathbf{y} \text{d} \mathbf{x})$. The Liouville vector field here is given by the radial field $\frac{1}{2}(\mathbf{x}\p_{\mathbf{x}} + \mathbf{y}\p_{\mathbf{y}})$, and the boundary $\p B^{2n} = S^{2n-1}$ with the induced contact structure $\xi_{\st} : = \ker \alpha_{\st}$, where $\alpha_{\st} = \lda_{\st}|_{S^{2n-1}}$, is called the \emph{standard contact sphere}. 


The standard contact sphere $(S^{2n-1}, \xi_{\st})$ has two special properties which are of our main interest in this paper.
First of all, it admits a \emph{dynamically convex} contact form $\alpha_{\st}$. Roughly speaking, this means that the periodic orbits of the associated Reeb dynamics satisfy a certain Conley--Zehnder index constraint, which can be seen as a symplectic counterpart of geometric convexity as in \cite{HWZ}. See Definition \ref{def: DC}. The other notable property of $(S^{2n-1}, \xi_{\st})$ is about the symplectic topology of fillings: The symplectic homology, which is a powerful symplectic invariant, of any Liouville filling vanishes by Seidel--Smith \cite[Corollary 6.5]{Sei08}. This for example includes the case of the standard symplectic ball. See Section \ref{sec: SH} for a brief description on symplectic homology and \cite{BO_MB} for more details.


Those properties seem significantly restrict symplectic and contact topology of Liouville domains and its contact boundary. To the authors knowledge, no other examples of fillable contact manifolds with the two properties are known, while there are examples of dynamically convex contact manifolds (in an appropriate sense) with fillings whose symplectic homology vanishes; e.g. see \cite{AbMa17}. Moreover, in \cite[Theorem 1.1]{Zh20}, it is proved in high dimensions that once a dynamically convex, even in a weak sense, contact manifold  admits a Liouville filling with vanishing symplectic homology, then every Liouville filling has vanishing symplectic homology.

The main question of this paper is the following.
\begin{question*}
How strongly do the dynamical convexity and the fillability characterize contact topological and dynamical aspects of the standard contact sphere?
\end{question*}

We remark that a related question was raised by Niederkrüger in the problem list \cite[Section 1.5]{AIM_Nie}. Our first result shows that, under a stronger fillability of contact manifolds, dynamical convexity of contact boundary indeed determines the contact structure:

\begin{TheoremA}\label{thm: result1}
    Let $(\Sigma,\xi)$ be a dynamically convex contact manifold of dimension $\geq 5$ with a flexible Weinstein filling $W$. If $\Sigma$ is simply-connected, then $(\Sigma, \xi)$ is contactomorphic to the standard contact sphere $(S^{2n-1}, \xi_{\st})$. Moreover the filling $W$ is Stein homotopy equivalent to the standard ball $B^{2n}$.
\end{TheoremA}

The main observation for the proof is that the dynamical convexity and the vanishing of symplectic homology force the domain to be a homology ball in high dimensions. Together with the simply-connectedness condition, the uniqueness of flexible Stein structures in \cite[Theorem 15.14]{CiEl12} tells us that the domain is in fact Liouville homotopy equivalent to $B^{2n}$ and hence the contact boundary is the standard contact sphere.

\begin{remark*}
In contrast, Theorem \ref{thm: result1} does not hold with a weaker version of dynamical convexity of contact manifolds. In \cite[Corollary 1.6]{Laz20}, it is proved that there are infinitely many dynamically convex, in the weaker sense, contact spheres which are flexibly fillable. See Remark \ref{rem: whyc1xi2}.
\end{remark*}



In the second part of the paper, we pay further attention to a dynamical aspect of the standard contact form $\alpha_{\st}$ on $S^{2n-1}$, namely, its Reeb flow is \emph{periodic}. It can directly be shown that the quantitative geometry from the periodicity of the Reeb flow on $S^{2n-1}$ is encoded in the filtered symplectic homology $\SH^a(B^{2n})$ of the standard ball in the following way: For each integer $m \geq 0$,
$$
\SH^{\epsilon + m T_P}(B^{2n}) \cong \Ho(B^{2n}, \p B^{2n})
$$
where $T_P$ is the period of the Reeb flow 
and $\epsilon>0$ is sufficiently small. Here the symplectic homology is filtered by the natural action filtration; see Section \ref{sec: actionfilt}. 
The following result shows that the periodic pattern in filtered symplectic homology from a periodic Reeb flow can still be observed for dynamically convex Liouville domains with vanishing symplectic homology.
Below, homology groups are over $\Z_2$.

\begin{TheoremA}\label{thm: result4}
 Let $(W, \lda)$ be a topologically simple Liouville domain with vanishing symplectic homology. Suppose that the natural contact form $\alpha = \lda|_{\p W}$ on the boundary yields a periodic Reeb flow of period $T_P$ and is dynamically convex. Then 
 $$
 \SH^{\epsilon + m T_P}(W) \cong \SH^{\epsilon + (m+2) T_P}(W)
 $$
for each integer $m \geq 0$. In particular, for $m \geq 0$ even, we further have 
$$
 \SH^{\epsilon + m T_P}(W) \cong \Ho(W, \p W).
$$

 If, in addition, $(\Sigma, \alpha)$ is Boothby--Wang, then
        $$
         \SH^{\epsilon + m T_P}(W) \cong \Ho(W, \p W) \cong \Ho(B^{2n}, \p B^{2n})
        $$
       for each integer $m \geq 0$, up to rescaling the radius of the ball $B^{2n}$.        
\end{TheoremA}

Above, we say a contact form is \emph{Boothby--Wang} if every Reeb orbit is periodic with the same minimal period; it is also referred to as a \emph{Zoll} contact form. 

It is worth noting that the above result is substantially motivated by a conjecture in Bowden--Crowley \cite{BoCr21}: \emph{Any compactly supported symplectomorphism on a flexible Weinstein domain is isotopic/homotopic to the identity}. In view of Floer theory for symplectomorphisms in \cite{Ul17}, Theorem \ref{thm: result4} serves as a positive clue for the conjecture with a certain class of symplectomorphisms called \emph{fibered twists}. See Section \ref{sec: conjBC} for details on this context.

As for the proof, we observe that the chain complex of symplectic homology is generated by periodic Reeb orbits, and their period defines a canonical action filtration on the complex. If the Reeb flow is periodic, then generators of the chain complex obviously appear in a periodic pattern along the action filtration. The dynamical convexity of the Reeb flow and the vanishing of symplectic homology allow us to compute the differential of the chain complex algebraically, and it turns out that the periodic pattern on the complex continues to hold at the homology level. Practically, we utilize Morse--Bott techniques in symplectic homology for periodic Reeb flows.

In the last part of the paper, we show conversely that filtered symplectic homology can sometimes characterize the periodicity of the Reeb flow on the contact boundary of Liouville domains in terms of symplectic capacities. In \cite{GH18}, a sequence of symplectic capacities $c_k^{\GH}$ with $k \in \N$, called \emph{Gutt--Hutchings capacities}, is defined via filtered positive equivariant symplectic homology for Liouville domains. 
The following theorem provides a sufficient and a necessary condition for dynamically convex Liouville domains with vanishing symplectic homology to admit a periodic or a Boothby--Wang Reeb flow in terms of the Gutt--Hutchings capacities. 



\begin{TheoremA} \label{thm: result3}
Let $(W, \lda)$ be a topologically simple Liouville domain $W$ with vanishing symplectic homology.
    \begin{enumerate}
        \item If the contact boundary $(\Sigma, \alpha)$ is Boothby--Wang of common period $T_P$ and is dynamically convex, then $c_1^{\GH}(W) = c_n^{\GH}(W) = T_P$.
        \item Suppose that $(\Sigma, \alpha)$ is Morse--Bott and dynamically convex. If $c_1^{\GH}(W) = c_n^{\GH}(W)$, then $(\Sigma, \alpha)$ is periodic. If, in addition, $W$ is a convex domain in $\R^{2n}$ containing the origin, then $(\Sigma, \alpha)$ is Boothby--Wang.
    \end{enumerate}
\end{TheoremA}

The above theorem can partially be seen as a symplectic homological version of the characterization result in \cite{GiGuMa21}, which might be of independent interest. In particular, an alternative proof may be given by proving the conjecture in \cite[Conjecture 1.9]{GH18} on the equivalence of the Gutt--Hutchings capacities and the Ekeland--Hofer capacities for convex domains in $\R^{2n}$. Another closely related characterization result for geometrically convex contact spheres with a periodic Reeb flow using the Ekeland--Hofer capacities can be found in \cite{MaRa20}.


An essential ingredient for the proof of Theorem \ref{thm: result3} is a Reeb orbit property in Proposition \ref{prop: MBReebprop} of the capacity $c_k^{\GH}(W)$ for Morse--Bott Liouville domains, which is particularly useful for computational purposes. It is a generalization of the one given in \cite[Theorem 1.1]{GH18} and allows us to determine $c_k^{\GH}(W)$ by looking at periodic Reeb orbits of a certain Conley--Zehnder index. 
In our case, the dynamical convexity serves as a strong constraint on the behavior of Conley--Zehnder indices under a standard perturbation  in \cite{Bo02} for Morse--Bott contact forms. As a result, we can get enough information on the Conley--Zehnder index of periodic Reeb orbits and their period to determine the capacity $c_k^{\GH}(W)$ by applying the Reeb orbit property.

\begin{remark*}
 Dynamical convexity and its consequences on the behavior of Conley--Zehnder indices of periodic Reeb orbits have provided fruitful applications on the multiplicity problem in Reeb dynamics. We refer the reader e.g. to \cite{AbMa17Ann, AbMa17, GiGu20, GiMa21} and references therein.
\end{remark*}




\subsection*{Organization of the paper} In Section \ref{sec: SH}, we provide with preliminaries on symplectic homology, focusing on a canonical action filtration on the chain complex and symplectic capacities from filtered (equivariant) symplectic homology. In Section \ref{sec: DC}, we discuss the notion of dynamical convexity of contact manifolds and its consequence for Morse--Bott type Reeb flows in Corollary \ref{cor: RSlowerbound} and filtered positive (equivariant) symplectic homology in Proposition \ref{prop: vanishingeSHlowdeg}. With those preliminaries, the proof of Theorem \ref{thm: result1} is given in Section \ref{sec: ThmA}. Theorem \ref{thm: result4} is proved in Section \ref{sec: ThmB}, and its relationship with a conjecture of Bowden--Crowley is discussed in Section \ref{sec: conjBC}. Section \ref{sec: ThmC} is devoted to prove Theorem \ref{thm: result3} where the main ingredient is the Morse--Bott Reeb orbit property in Proposition \ref{prop: MBReebprop}.

\section{Symplectic homology}  \label{sec: SH}

\subsection{Filtered symplectic homology} \label{sec: fSH} We briefly recall basic notions in (equivariant) symplectic homology for Liouville domains. We mainly follow \cite{BO_MB, Gutt}. See also \cite{CiOa18}.

\subsubsection{Admissible Hamiltonians} Let $(W, \lda)$ be a Liouville domain. 
The boundary $\p W$ admits a canonical contact structure $\xi = \ker \alpha$ where $\alpha = \lda|_{\p W}$. By attaching a part of symplectization $[1, \infty) \times \p W$ along the boundary $\p W$ through the Liouville flow, we complete the domain $W$ as 
\[
\widehat{W} = W \cup_{\p W} ([1, \infty) \times \p W), \quad \widehat{\lda}= \lda \cup_{\p W} (r \alpha)
\]
where $r$ denotes the coordinate of $[1, \infty)$. A time-dependent Hamiltonian $H: S^1 \times \widehat W \rightarrow \R$ is called  \emph{admissible} if Hamiltonian 1-periodic orbits are nondegenerate, $H$ is negative and $C^2$-small (and Morse) in the interior $W \subset \widehat{W}$, and $H$ is of the form $\mathfrak{a}r+\mathfrak{b}$ at the end, where $\mathfrak{a} >0, \mathfrak{b} \in \R$. We call $\mathfrak{a}$ the \emph{slope} of $H$, and we assume that $\mathfrak{a}$ is not the period of a periodic Reeb orbit on the contact boundary $(\p W, \alpha)$.  We denote by $\mathcal{P}(H)$ the set of \emph{contractible} Hamiltonian 1-periodic orbits. To each $\gamma \in \mathcal{P}(H)$ we associate the \emph{Conley--Zehnder index} $\CZ(\gamma) \in \Z$ as in \cite[Section 2]{BO_MB}, which is well-defined due to the assumption that $c_1(TW) = 0$. In this paper the Hamiltonian vector field $X_H$ is defined by $\iota_{X_H}\ow  = dH$.

\begin{remark}\label{rem: globalassumption}
Throughout the paper, we will often assume that  $(W, \lda)$ is \emph{topologically simple}, i.e. $c_1(TW) = 0$ and the inclusion $\p W \rightarrow W$ is $\pi_1$-injective. In particular, the condition $c_1(TW) =0$ is to give a $\Z$-grading via Conley--Zehnder index in symplectic homology. Note that if $W$ is a Weinstein domain of dimension $2n$ with $n \geq 3$, then it is $\pi_1$-injective. See also Remark \ref{rem: whyc1xi}.


\end{remark}

\subsubsection{Chain complex} Let $\F$ be an arbitrary field. Take an \emph{admissible} $S^1$-family of $d \widehat{\lda}$-compatible almost complex structures $J$ on $\widehat W$ in the sense of \cite[Section 2]{BO_MB}. For the pair $(H, J)$, we define a chain complex $\CF_*(H, J)$ to be a $\Z$-graded vector space over $\F$ by
$$
\CF_k(H, J) = \bigoplus_{\substack{\gamma \in \mathcal{P}(H) \\ |\gamma| = k}} \F \langle \gamma \rangle
$$
where the grading is defined by $|\gamma| : =  - \CZ(\gamma)$. As fairly standard in Floer theory, the differential 
$$
\p: \CF_k(H, J) \rightarrow \CF_{k-1}(H, J)
$$
is defined by counting Floer trajectories connecting 1-periodic orbits, which is well-defined for a generic pair $(H, J)$. The homology of the chain complex $(\CF_*(H, J), \p)$ is denoted by $\HF_*(H, J)$, called the \emph{Hamiltonian Floer homology}. Hamiltonian Floer homology groups form a directed system via increasing the slope $\mathfrak{a}$ of admissible Hamiltonians, and the \emph{symplectic homology} $\SH_*(W)$ is defined to be the direct limit:
$$
\SH_*(W) : = \varinjlim_{\mathfrak{a}} \HF_*(H, J).
$$

\begin{remark}\label{rem: MBpertadmHamcorresp}
Let $W$ be a nondegenerate Liouville domain with $c_1(TW) = 0$, meaning that the contact form $\alpha$ is nondegenerate. Then \emph{time-independent} admissible Hamiltonians $H$ are transversely nondegenerate and Hamiltonian 1-periodic orbits appear in Morse--Bott $S^1$-families; see \cite[Lemma 3.3]{BO_MB}. Each $S^1$-family corresponds to a Reeb orbit $\gamma$. There is a standard way to perturb $H$ to a time-dependent admissible Hamiltonian $H_{\delta}$, for example as in \cite[Equation (25)]{BO_MB}, and the Morse--Bott $S^1$-families $\gamma$ of $H$ split into two copies of 1-periodic orbits $\hat{\gamma}$, $\check{\gamma}$ of $H_{\delta}$. The relation between the Conley--Zehnder indices of those orbits are given by 
$$
\CZ(\hat \gamma) = \CZ(\gamma) - 1, \; \CZ(\check \gamma) = \CZ(\gamma)
$$
where $\CZ(\gamma)$ is the Conley--Zehnder index of the Reeb orbit $\gamma$ \cite[Lemma 3.4]{BO_MB}.
\end{remark}


\subsubsection{Action filtration} \label{sec: actionfilt}The Floer chain complex $\CF_*(H, J)$ admits a canonical filtration via the following action functional on the free loop space of $\widehat W$:
$$
\mathcal{A}_H(\gamma) = - \int_{S^1} \gamma^* \lda  - \int_0^1 H(t, \gamma(t))dt.
$$
The value $\mathcal{A}_H(\gamma)$ is decreasing along Floer trajectories, and it yields a filtration on the chain complex $\CF_*(H, J)$ as
$$
\CF_*^a(H, J) : = \bigoplus_{\substack{\gamma \in \mathcal{P}(H) \\ |\gamma| = k \\ \mathcal{A}_H(\gamma) < a} } \F \langle \gamma \rangle
$$
for every $a \in \R$. Accordingly, we define the \emph{filtered symplectic homology} $\SH_*^a(W)$ by the direct limit of the filtered Hamiltonian Floer homology groups:
$$
\SH_*^a(W) : = \varinjlim_{\mathfrak{a}} \HF_*^a(H, J).
$$
By the assumption that $H$ is $C^2$-small and Morse on $W$, we can take a sufficiently small number $\epsilon >0$ such that there is a canonical identification
\begin{equation}\label{eq: SHandHo}    
\SH_*^{\epsilon}(W) \cong \Ho_{*+n}(W, \p W).
\end{equation}
See e.g. \cite[Proposition 2.10]{CiOa18}.

Let $a < b$ be real numbers. The quotient complex $\CF_*^{(a, b)}(H, J)$ of the action window $(a, b)$ is defined by
$$
\CF_*^{(a, b)}(H, J)  := \CF_*^{b}(H, J)/ \CF_*^{a}(H, J).
$$
The corresponding symplectic homology  of the action window is denoted by $
\SH_*^{(a, b)}(W)$. In particular, we call 
$$
\SH_*^+(W) : = \SH_*^{(\epsilon, \infty)}(W), \quad \SH_*^{+, a}(W) : = \SH_*^{(\epsilon, a)}(W)
$$ 
the \emph{(filtered) positive symplectic homology}, where $\epsilon$ is chosen as in \eqref{eq: SHandHo} and $a > \epsilon$. We have the following useful tautological long exact sequences \cite[Equation (6)]{CiOa18}.

\begin{proposition}\label{prop: lesinSH}
For $a < b< c$, there exists a long exact sequence on filtered symplectic homology
\begin{equation}\label{eq: tau_LES}
\rightarrow \SH_k^{(a, b)}(W) \xrightarrow{j} \SH_k^{(a, c)}(W) \rightarrow \SH_k^{(b, c)}(W) \xrightarrow{\iota} \SH_{k-1}^{(a, b)}(W) \rightarrow.
\end{equation}
In particular, under the identification \eqref{eq: SHandHo},
\begin{equation}\label{eq: tau_LES_Ho}
\rightarrow \Ho_{k + n}(W, \p W) \xrightarrow{j^a} \SH^a_k(W) \rightarrow \SH_k^{+, a}(W) \xrightarrow{\iota^a} \Ho_{k + n-1}(W, \p W) \rightarrow
    \end{equation}
    for each $a > \epsilon$.
\end{proposition}


    


\subsection{Equivariant symplectic homology} \label{sec: eSH} A variant of symplectic homology is its equivariant version, called \emph{(filtered) equivariant symplectic homology}, denoted by $\SH_*^{S^1, a}(W)$. It is defined by the fact that $S^1$ acts on the loop space of $\widehat W$ via the translation on the domain of loops. In this paper, we use the version described in \cite[Section 2.1]{Gutt}. It is a $\Z$-graded vector space over $\F$ and admits a canonical action filtration \cite[Proposition 2.11]{Gutt}. The positive action part is denoted by $\SH_*^{S^1, +, a}(W)$. The following are analogous properties to the non-equivariant version which we will use later. 

\begin{enumerate}
\item For a small $\epsilon > 0$ as in \eqref{eq: SHandHo}, we have a canonical identification \cite{CiOa18, Vit} 
$$
\SH_*^{S^1, \epsilon}(W) \cong \Ho^{S^1}_{*+n}(W, \p W)
$$
where $S^1$ acts trivially on $W$.
\item (Tautological exact sequence \cite[Proposition 6.5]{GH18}) There is a long exact sequence 
$$
\rightarrow \Ho_{k + n}^{S^1}(W, \p W) \xrightarrow{j^{S^1, a}} \SH^{S^1, a}_k(W) \rightarrow \SH_k^{S^1, +, a}(W) \xrightarrow{\iota^{S^1, a}} \Ho_{k + n-1}^{S^1}(W, \p W) \rightarrow
$$
for each $a > \epsilon$.
\item (Gysin exact sequence \cite{BO_Gysin, CiOa18}) There is a long exact sequence
$$
 \rightarrow \SH_k^{+}(W) \xrightarrow{g^+} \SH^{S^1, +}_{k} (W) \xrightarrow{U} \SH^{S^1, +}_{k-2} (W) \xrightarrow{\delta} \SH_{k-1}^{+}(W) \rightarrow
$$
where we call the middle map $U: \SH^{S^1, +}_{k} (W) \rightarrow \SH^{S^1, +}_{k-2} (W)$ the \emph{$U$-map} following \cite{GH18}. The Gysin exact sequence is compatible with the action filtration and hence compatible with the tautological exact sequence.
\end{enumerate}

\subsection{Symplectic homology capacities}\label{sec: shcapacities} Let $(W, \lda)$ be a Liouville domain. There are several symplectic capacities which can be defined from filtered symplectic homology. One of them is the \emph{symplectic homology capacity} $c^{\SH}(W)$ given as follows.
$$
c^{\SH}(W) := \inf \{a > 0 \;|\; j^a[W, \p W] = 0\} \in [0, \infty]
$$
where the map $j^a: \Ho_*(W, \p W) \rightarrow \SH^{a}(W)$ is of the tautological exact sequence \eqref{eq: tau_LES_Ho}. Conventionally, $c^{\SH}(W) = \infty$ if $j^a[W, \p W] \neq 0$ for every $a > 0$. The capacity $c^{\SH}(W)$ is finite if and only if $\SH(W) = 0$. For detailed discussions and applications of $c^{\SH}(W)$, we refer the reader to \cite{AbKa22, Irie, KiKiKw22}.

\begin{remark}\label{rem: eqdefcSH}
By the tautological exact sequence  \eqref{eq: tau_LES_Ho}, the capacity $c^{\SH} (W)$ can also be defined by
$$
c^{\SH}(W) = \inf \{a >  0 \;|\; \exists \; c \in \SH_{n+1}^{+, a}(W) \text{ such that $\iota^{a}(c) = [W, \p W]$} \}.
$$
\end{remark}

In \cite[Section 4]{GH18}, a sequence of symplectic capacities $c_k^{\GH}(W)$ for $k \in \N$, called the \emph{Gutt--Hutchings capacities}, is defined using filtered positive equivariant symplectic homology $\SH^{S^1, + , a}(W)$. For a nondegenerate Liouville domain $W$, the capacity $c_k^{\GH}(W) \in [0, \infty]$ is defined by
$$
c_k^{\GH}(W) : = \inf\{a > 0\;|\; \exists \; c \in \SH_{2k + n-1}^{S^1, +, a}(W) \text{ such that $(\iota^{S^1} \circ U^{k-1} \circ \delta^{+, a})(c) = [W, \p W]^{S^1}$}\}
$$ 
where $ \delta^{+, a}:  \SH_{2k+n-1}^{S^1, +, a}(W) \rightarrow \SH_{2k+n-1}^{S^1, +}(W)$ is the natural inclusion and the class $[W, \p W]^{S^1} \in \Ho^{S^1}(W, \p W)$ denotes the equivariant fundamental class of $(W, \p W)$. For later use, we make the following simple observation.
\begin{proposition}\label{prop: eqdefc1gh}
The first capacity $c_1^{\GH}(W)$ can equivalently be given by
$$
c_1^{\GH}(W) = \inf \{a > 0 \;|\;  j^{S^1, a}[W, \p W]^{S^1} = 0\}
$$
where $j^{S^1, a}$ is the map in the tautological exact sequence.
\end{proposition}
\begin{proof}
For convenience, let us temporarily denote
$$
c(W) := \inf \{a > 0 \;|\;  j^{S^1, a}[W, \p W]^{S^1} = 0\}.
$$
Consider the following commutative diagram.
$$
\begin{tikzcd} 
\SH^{S^1, a}(W) \arrow[r]  \arrow[d, "\delta^a"]&  \SH^{S^1, +, a}(W) \arrow[r, "\iota^{S^1, a}"] \arrow[d, "\delta^{+, a}"] & \Ho^{S^1}(W, \p W) \arrow[r, "j^{S^1, a}"] \arrow[d, "\id"] & \SH^{S^1, a}(W)  \arrow[d, "\delta^{a}"] \\
\SH^{S^1}(W) \arrow[r] & \SH^{S^1, +}(W) \arrow[r, "\iota^{S^1}"] & \Ho^{S^1}(W, \p W) \arrow[r,  "j^{S^1}"] & \SH^{S^1}(W)
\end{tikzcd}
$$
Here the vertical maps are the natural inclusions, and the horizontal maps are from the tautological exact sequence in Section \ref{sec: eSH}. (Just for simplicity we omitted the grading.) Suppose that there exists $c \in \SH^{S^1, +, a}(W)$ such that $(\iota^{S^1}\circ \delta^{+, a})(c) = [W, \p W]^{S^1} \in  \Ho^{S^1}(W, \p W)$. Then by commutativity we have
$$
\iota^{S^1, a} (c) = (\iota^{S^1}\circ \delta^{+, a})(c) = [W, \p W]^{S^1}.
$$
It follows that
$$
j^{S^1, a}[W, \p W]^{S^1} = (j^{S^1, a} \circ \iota^{S^1, a}) (c) = 0.
$$
This implies that $c_1^{\GH}(W) \geq c(W)$. The opposite inequality $c_1^{\GH}(W) \leq c(W)$ can also be shown by a similar diagram chasing.
\end{proof}



For an arbitrary (possibly degenerate) Liouville domain $W$, the capacity $c_k^{\GH}(W)$ is defined in \cite[Section 4.2]{GH18} by perturbing the domain to be nondegenerate: For a smooth function $f: \p W \rightarrow \R$ on the contact boundary $\p W$, the \emph{graphical subdomain} $W_f$ of the completion $\widehat W$ is defined by
$$
W_f : = \{(r, y) \in \R \times \p W \subset \widehat W \;|\; f(y)\geq r\}
$$
where $\R \times \p W$ is embedded in $\widehat W$ via the Liouville flow. Then we have that $\widehat{\lda}|_{\p W_f} = f \alpha$ under the identification 
$$
\p W \rightarrow \p W_f, \quad y \mapsto (f(y), y) \in \R \times \p W.
$$
The domain $W_f$ is nondegenerate for generic $f$. As in \cite[Definition 4.4]{GH18}, we define
$$
c_k^{\GH}(W) := \inf_{f >1} c_k^{\GH}(W_f) = \sup_{f < 1} c_k^{\GH}(W_f)
$$
where the infimum and the supremum are taken over nondegenerate $W_f$.

\begin{remark}
Conjecturally, the Gutt--Hutchings capacities agree with the Ekeland--Hofer capacities \cite{EHcapacity} for starshaped domains $W \subset \R^{2n}$. See \cite[Conjecture 1.9]{GH18}.
\end{remark}

We will use the following axiomatic properties of the capacity $c_k^{\GH}(W)$ from \cite[Proposition 4.6]{GH18}.

\begin{proposition} \label{prop: axiomaticproportyofc_i}
Let $(W, \lda)$ be a topologically simple Liouville domain. Then the capacity $c_k^{\GH}(W)$ satisfies the following properties.
\begin{enumerate}
\item (Monotonicity) If $W'$ is a subdomain of $W$, then 
$$
c_k^{\GH}(W') \leq c_k^{\GH}(W). 
$$
\item (Reeb orbit) If $c_k^{\GH}(W) < \infty$, then $c_k^{\GH}(W) = \mathcal{A}(\gamma)$ for some Reeb orbit $\gamma$ in $(\p W, \alpha)$. If, in addition, $W$ is nondegenerate, then the Conley-Zehnder index of $\gamma$ is given by
\begin{equation}\label{eq: CZindReeborbitpropertynondeg}
\CZ(\gamma) = 2k + n-1.
\end{equation}
\end{enumerate}
\end{proposition}

\begin{remark}
The Conley--Zehnder index condition \eqref{eq: CZindReeborbitpropertynondeg} of the Reeb orbit property for nondegenerate Liouville domains is not explicitly mentioned in \cite{GH18}, but it follows directly from their proof for the starshaped domains \cite[Theorem 1.1]{GH18} where the positive equivariant symplectic homology and the maps involved are likewise $\Z$-graded.
\end{remark}

\section{Dynamically convex contact manifolds} \label{sec: DC}

Let $(\Sigma, \xi)$ be a contact manifold of dimension $2n-1$ with $c_1(\xi) = 0$. For a contact form $\alpha$, denote the set of contractible periodic Reeb orbits of the corresponding Reeb flow by $\mathcal{P}(\alpha)$. We assign the \emph{Conley--Zehnder index} $\CZ(\gamma)$ to each nondegenerate contractible Reeb orbit $\gamma \in \mathcal{P}(\alpha)$ in the standard way, e.g. as in \cite[Section 2]{BO_MB} and \cite[Section 4]{GiGu20}: Taking a capping disk of $\gamma$, we have a symplectic trivialization of the contact structure $\xi$ over $\gamma$. With respect to this trivialization, obtain a path of symplectic matrices which represent the linearized Reeb flow restricted to $\xi$. We then set $\CZ(\gamma)$ to be the Conley--Zehnder index of the path. This is independent of the choice of trivializations due to the assumption $c_1(\xi) = 0$. A nondegenerate contact form $\alpha$ is called \emph{dynamically convex} if $\CZ (\gamma) \geq n+1$ for any $\gamma \in \mathcal{P}(\alpha)$. 

For a general (not necessarily nondegenerate) contact form, we define the notion of dynamical convexity in terms of an extension of the Conley--Zehnder index for degenerate orbits, called the \emph{lower Conley--Zehnder} index $\CZ_-(\gamma)$. Basically, $\CZ_-(\gamma)$ is defined to be the infimum of Conley--Zehnder indices of  nondegenerate nearby perturbations. See \cite{GiGu20, booklong} and references therein for a precise definition.

\begin{definition}\label{def: DC}
A contact form $\alpha$ is \emph{dynamically convex} if 
$$
\CZ_{-} (\gamma) \geq n+1
$$
for any contractible periodic orbit $\gamma \in \mathcal{P}(\alpha)$. A Liouville domain $(W, \lda)$ is called \emph{dynamically convex} if the induced contact form $\alpha = \lda|_{\p W}$ is dynamically convex.
\end{definition}

\begin{remark}\label{rem: whyc1xi}\
\begin{enumerate}
\item For a Liouville domain $W$, note that $i^* c_1(TW) = c_1(\xi)$ where $i : \p W \rightarrow W$ is the inclusion and $\xi$ is the canonical contact structure on $\p W$. Therefore if $c_1(TW) = 0$ as for topologically simple Liouville domains, we have $c_1(\xi) = 0$.

\item A contact manifold admits a periodic Reeb orbit $\gamma$ which is contractible \emph{in} $W$ if there is a Liouville filling with vanishing symplectic homology; see \cite[Section 4]{Vit}. If the inclusion $i: \p W \rightarrow W$ is $\pi_1$-injective, it follows that $\gamma$ is also contractible \emph{in} $\p W$. Notice that this makes the notion of dynamical convexity nontrivial. 

\end{enumerate}
\end{remark}

\begin{remark}\label{rem: whyc1xi2}\
\begin{enumerate}
\item There are several related notions of `dynamical convexity'. Our definition is the same as the one in \cite[Definition 4.9]{GiGu20}, and a more general definition using contact homology can be found in \cite[Section 2.1]{AbMa17Ann}. 

Other closely related but weaker versions of dynamical convexity appear in \cite{CiOa18, Laz20, Ue19}. It requires that $\CZ(\gamma) > 3-n$ for any contractible Reeb orbits $\gamma$, which morally means that the contact homology is trivial in negative degrees. This notion is also referred to as \emph{index-positivity} of contact manifolds as in \cite[Section 9.5]{CiOa18}.

\item  By the work of Hofer--Wysocki--Zehnder \cite{HWZ} (see also \cite[Theorem 4.10]{GiGu20}), geometrically convex domains in $\R^{2n}$ are dynamically convex, but the converse is not true as discovered in \cite{ChEd22, DaGuZh22}.
\end{enumerate}
\end{remark}

In our applications, we will mainly deal with Morse--Bott Liouville domains, i.e. the case when the induced contact form on the boundary is Morse--Bott. For Morse--Bott contact forms, the dynamical convexity can be understood in terms of another extension of the Conley--Zehnder index for degenerate orbits: Denote the Morse--Bott family of period $T$ by
$$
\Sigma_T : = \{z \in \Sigma \;|\; \phi_{\alpha}^T(z) = z\}.
$$
To each possibly degenerate Reeb orbit $\gamma$ we can associate an index $\RS(\gamma) \in \frac{1}{2}\Z$, called the \emph{Robbin--Salamon index}, introduced in \cite{RS}. Every member $\gamma$ of a Morse--Bott family $\Sigma_T$ has the same Robbin--Salamon index by the homotopy property in \cite[Theorem 2.4]{RS}, and we denote it by $\RS(\Sigma_T)$.


\begin{proposition}\label{prop: RSlowerbound}
Let $\alpha$ be a Morse--Bott contact form on a closed contact manifold $\Sigma$ of dimension $2n-1$. Then $\alpha$ is dynamically convex if and only if
$$
\RS(\Sigma_T) - \frac{1}{2}(\dim \Sigma_T - 1) \geq  n+1 
$$
for any Morse--Bott family $\Sigma_T$.
\end{proposition}

\begin{proof}
Let $\Sigma_T$ be a Morse--Bott family of period $T$. The standard perturbation method in \cite[Section 2.2]{Bo02} yields that the Morse--Bott family $\Sigma_T$ splits into nearby nondegenerate Reeb orbits, and the minimum of Conley--Zehnder indices of them is given by
$$
\RS(\Sigma_T) - \frac{1}{2}(\dim \Sigma_T - 1).
$$
It follows that $\RS(\Sigma_T) - \frac{1}{2}(\dim \Sigma_T - 1) \geq n+1$ if $\alpha$ is dynamically convex.

Conversely, assume that $\RS(\Sigma_T) - \frac{1}{2}(\dim \Sigma_T - 1) \geq n+1$ for any Morse--Bott family $\Sigma_T$. The behavior of the Robbin--Salamon index (and the Conley--Zehnder index) under perturbations of contact forms is studied in \cite[Section 4]{Mc16}; see Remark \ref{rem: McLeanusesCZ}. In particular, by the relationship on Conley--Zehnder indices under $C^2$-perturbations in \cite[Lemma 4.10]{Mc16}, it follows that $C^2$-close nondegenerate perturbations $\tilde \gamma$ nearby $\Sigma_T$ have Conley--Zehnder index 
$$
\CZ(\tilde \gamma) \geq \RS(\Sigma_T) - \frac{1}{2}(\dim \Sigma_T - 1).
$$
The lower Conley--Zehnder index is given by the infimum of the Conley--Zehnder index of nearby nondegenerate perturbations, so it follows that
$$
\CZ_-(\gamma) \geq \RS(\Sigma_T) - \frac{1}{2}(\dim \Sigma_T - 1) \geq n+1
$$
for any Reeb orbit $\gamma$ of $(\Sigma, \alpha)$. Therefore $\alpha$ is dynamically convex.
\end{proof}

\begin{remark}\label{rem: McLeanusesCZ}
For the sake of clarification, we remark that the Robbin--Salamon index for degenerate periodic Reeb orbits is referred to as the Conley--Zehnder index in \cite[Section 4.3]{Mc16}.
\end{remark}

\begin{remark}\label{rem: lCZandthelowerRS}
The above proof actually shows that, if $\gamma$ is a member of a Morse--Bott family $\Sigma_T$, then
$$
\CZ_-(\gamma) = \RS(\Sigma_T) - \frac{1}{2}(\dim \Sigma_T - 1).
$$
\end{remark}

The following is a direct consequence of Proposition \ref{prop: RSlowerbound}.

\begin{corollary}\label{cor: RSlowerbound}
Let $\alpha$ be a Morse--Bott contact form on a closed contact manifold $\Sigma$ of dimension $2n-1$. 
\begin{enumerate}
\item If $\RS(\Sigma_T) \geq 2n$ for every Morse--Bott family $\Sigma_T$, then $\alpha$ is dynamically convex.

\item Suppose that the Reeb flow of $\alpha$ is in addition periodic and is dynamically convex. Then $\RS(\Sigma_{T_P}) \geq 2n$ where $T_P$ is a common period.
\end{enumerate}
\end{corollary}

\begin{example}\label{ex: ballexampleDC}
Consider the $2n$-dimensional closed ball $B^{2n}$ as a Liouville domain equipped with the standard Liouville form
$$
\lda_{\st} :  = \frac{i}{2}\sum_{j =1}^n (z_jd\overline z_j - \overline z_j d z_j)
$$
where $(z_1, \dots, z_n)$ is the complex coordinates in $\C^n$.
Denote the induced contact form on the boundary $\p B^{2n} = S^{2n-1}$ by
$$
\alpha_{\st} : = \lda_{\st}|_{S^{2n-1}}.
$$
The Reeb flow of $\alpha_{\st}$ is given by the coordinate-wise rotations
\begin{equation}\label{eq: perflowonsphere}
(z_1, \dots, z_n) \longmapsto (e^{it}z_1 \dots, e^{it}z_n).
\end{equation}
It follows that every Reeb orbit is periodic with the same period $2\pi$. In particular $\alpha_{\st}$ is of Boothby--Wang type and hence is degenerate. The Robbin--Salamon index of each Morse--Bott family $\Sigma_{2\pi m}$, $m \geq 1$, is given by
$$
\RS(\Sigma_{2 \pi m}) = 2n m \geq 2n.
$$
See e.g. \cite[Section 4.1.2]{Bou_Lec} or \cite[Section 3]{Oa04}. Therefore $\alpha_{\st}$ is dynamically convex by Corollary \ref{cor: RSlowerbound}.
\end{example}

\begin{remark}
There is also a nondegenerate dynamically convex contact form on the standard contact sphere $S^{2n-1}$. Let $a_1, \dots, a_n$ are rationally independent positive numbers. Define
$$
\alpha_a : = \frac{i}{2}\sum_{j =1}^n a_j(z_jd\overline z_j - \overline z_j d z_j).
$$
Due to the rational independence, $\alpha_a$ is nondegenerate, and direct computations of Conley--Zehnder indices show that it is dynamically convex. See \cite[Section 4.1.1]{Bou_Lec}.
\end{remark}

Later, we will use the following implications of dynamical convexity for filtered symplectic homology, which are fairly well-known; see \cite[Section 3.3.1]{Laz20} for a closely related observation.

\begin{proposition}\label{prop: vanishingeSHlowdeg}
Let $(W, \lda)$ be a nondegenerate dynamically convex Liouville domain with $c_1(TW) = 0$. Then the positive symplectic homology groups $\SH_k^{+, a}(W)$ and $\SH^{S^1, + , a}_{k} (W)$ are vanishing for $k < n+1$ and $a \in \R$. 
\end{proposition}

\begin{proof}
As in \cite[p. 3587]{GH18}, each nontrivial generator of the positive equivariant symplectic homology $\SH^{S^1, +, a}_k(W)$ corresponds to a Reeb orbit $\gamma$ with $\CZ(\gamma) = k - 2m$ for some integer $m \geq 0$. Since $\CZ(\gamma) \geq n+1$ by the dynamical convexity, it follows that
$$
k \geq \CZ(\gamma) \geq n+1.
$$
Therefore $\SH^{S^1, +, a}_k(W) = 0$ for $k < n+1$.

As for the non-equivariant case, consider the following filtered Gysin exact sequence in symplectic homology \cite{BO_Gysin, CiOa18}:
\begin{equation}\label{eq: GysininSH}
\rightarrow \SH^{S^1, + , a}_{k-1} (W) \rightarrow \SH_k^{+, a}(W) \rightarrow \SH^{S^1, + , a}_{k} (W) \rightarrow \SH^{S^1, + , a}_{k-2} (W) \rightarrow  \SH_{k-1}^{+, a}(W) \rightarrow
\end{equation}
It follows directly from the exactness that
$$
\SH^{+, a}_k(W) = 0
$$
for $k < n+1$.
\end{proof}

\begin{remark}
In the above proof, one can directly show that $\SH^{+, a}_k(W) = 0$, instead of applying the Gysin exact sequence, by using the correspondence between the generators of $\SH^{+, a}(W)$ and Reeb orbits on the boundary $(\p W, \alpha)$. See Remark \ref{rem: MBpertadmHamcorresp} and \cite[Section 3.3.1]{Laz20}.
\end{remark}




\section{Contactomorphism type of dynamically convex contact manifolds}\label{sec: ThmA}

In this section we give a proof of Theorem \ref{thm: result1}. We start with the following observation on the relative homology group.

\begin{proposition}\label{prop: homologyball}
    Let $(W, \lda)$ be a nondegenerate and dynamically convex Liouville domain with $c_1(TW) = 0$, and suppose that $\SH_*(W) = 0$. Then $\Ho_*(W, \p W) \cong \Ho_*(B^{2n}, \p B^{2n})$ over any field $\F$.
\end{proposition}

\begin{proof}
    By the long exact sequence in Proposition \ref{prop: lesinSH} with $a = -\infty, b = \epsilon, c= \infty$, it follows from the assumption $\SH_*(W) = 0$ that
    \begin{equation} \label{eq: pSHandHo}
        \SH_*^{+}(W) \cong \Ho_{*+n-1}(W, \p W).
    \end{equation}
    Here we applied the identification \eqref{eq: SHandHo}. Since $(\p W, \alpha)$ is dynamically convex, it follows that
    \[
\SH_k^{+}(W) = 0
    \]
    for $k < n+1$ by Proposition \ref{prop: vanishingeSHlowdeg}. Therefore $ \Ho_k(W, \p W) = 0$ for $k < 2n$. Note that $W$ is a compact manifold with boundary, so we know that $\Ho_{2n}(W, \p W) \cong \F$, generated by the fundamental class. Summing up, we conclude that
    $$
    \Ho_*(W, \p W) \cong \Ho_*(B^{2n}, \p B^{2n})
    $$
    as asserted.
\end{proof}

\begin{lemma} \label{lem: homologyball}
Let $(W, \lda)$ be a dynamically convex Liouville domain with $c_1(TW)=0$ and $\SH_*(W; \F) =0$ for any field $\F$. Then $W$ is a homology ball over $\Z$.
\end{lemma}

\begin{proof}
    By Proposition \ref{prop: homologyball} we know that
    \[
    \Ho_*(W, \p W; \F) \cong \Ho_*(B^{2n}, \p B^{2n}; \F)
    \]
    for any field $\F$.
    In this case the universal coefficient theorem yields
    \[
    \Ho_*(W) \cong \Ho_*(B^{2n})
    \]
 over $\Z$. See for example \cite[Excercise 3.A.3]{Hat02}.
\end{proof}

\begin{remark}
It is also possible to work directly over $\Z$-coefficients in Propositions \ref{prop: vanishingeSHlowdeg} and \ref{prop: homologyball}. In this case Lemma \ref{lem: homologyball} will be superfluous for the proof of Theorem \ref{thm: result1} below.
\end{remark}

\begin{lemma}\label{lem: simply-conn}
    Let $W$ be a Weinstein domain of dimension $2n \geq 6$. If the boundary $\p W$ is simply-connected, then so is $W$.
\end{lemma}

\begin{proof}
The domain $W$ can be constructed from $\p W$ by attaching handles of index greater than or equal to $ n \geq 3$. This does not affect the fundamental group of $\p W$. Therefore $W$ is still simply-connected.
\end{proof}

Now we give a proof of Theorem \ref{thm: result1}.

\begin{proof}[Proof of Theorem \ref{thm: result1}]
    Combining Lemmas \ref{lem: homologyball} and \ref{lem: simply-conn}, the $h$-cobordism theorem tells us that the Weinstein filling $W$ is diffeomorphic to the ball $B^{2n}$ and hence the boundary $\Sigma$ is diffeomorphic to the sphere $S^{2n-1}$.

    It is well-known that every almost complex structure $J$ on $B^{2n}$ is homotopic to the standard one $J_{\st}$. In this case the result on uniqueness of flexible Stein structure in \cite[Theorem 15.14]{CiEl12} shows that $(W, J, \phi)$ is in fact Stein homotopy equivalent to $(B^{2n}, J_{\st}, \phi_{\st})$. It follows that, with respect to the induced Liouville form $\lda : = - d \phi \circ J$, the Liouville domain $(W, \lda)$ is Liouville homotopy equivalent to $(B^{2n}, \lda_{st})$; see \cite[Definition 11.5]{CiEl12}. Consequently the Gray stability implies that the contact boundary $(\Sigma, \xi)$ is contactomorphic to the standard sphere $(S^{2n-1}, \xi_{\st})$.
\end{proof}

\begin{remark}
    By the theorem of Eliashberg--Floer--McDuff \cite{McD91}, every symplectically aspherical symplectic filling of the standard contact sphere $(S^{2n-1}, \xi_{\st})$, $n \geq 3$, is diffeomorphic to the ball $B^{2n}$. The proof of Theorem \ref{thm: result1} shows further that flexible Weinstein fillings of $(S^{2n-1}, \xi_{\st})$, $n \geq 3$, are unique up to Stein homotopy.
\end{remark}

\section{Periodicity in filtered symplectic homology} \label{sec: ThmB}

\subsection{Proof of Theorem \ref{thm: result4}} We start with the following observation on the symplectic homology capacity.

\begin{lemma} \label{lem: cSHforperflow}
Let $(W, \lda)$ be a Liouville domain with vanishing symplectic homology and $c_1(TW) = 0$. Suppose that the induced contact form $\alpha$ on the boundary yields a periodic Reeb flow of period $T_P$ and it is dynamically convex. Then the symplectic homology capacity satisfies
$$
c^{\SH}(W) \leq T_P.
$$
As a directly consequence, if we further assume that the Reeb flow is of Boothby--Wang type, then 
$$
c^{\SH}(W) = T_P.
$$
\end{lemma}

\begin{proof}

We first claim that the strict inequality
\begin{equation}\label{eq: minCZ>n+1}
\RS(\Sigma_T) -\frac{1}{2}(\dim \Sigma_T  -1) > n+1
\end{equation}
holds for $T > T_P$. For this, we employ a property of lower Conley--Zehnder indices $\CZ_-(\gamma)$ of periodic Reeb orbits $\gamma$ from \cite[Lemma 4.8]{GiGu20}, namely, for dynamically convex Reeb flows the index $\CZ_-(\gamma)$ strictly increases under iteration i.e.
$$
\CZ_-(\gamma^k) < \CZ_-(\gamma^{k'})
$$
if $k < k'$. In our setup, let $\gamma_T$ be an orbit in a Morse--Bott family $\Sigma_T$ with $T > T_P$. Then since the Reeb flow is periodic, there must be an orbit $\gamma_0$ such that the $k$-th iteration $\gamma_0^k$ is in the family $\Sigma_{T_P}$ and $\gamma_0^{k'} = \gamma_T$ for some $k < k'$. It follows that
\begin{equation}\label{eq: iterCZ-inequality}
\CZ_-(\gamma_T) = \CZ_-(\gamma_0^{k'})  > \CZ_-(\gamma_0^{k})
\end{equation}
by the increasing property under iteration.
As pointed out in Remark \ref{rem: lCZandthelowerRS}, the lower Conley--Zehnder index $\CZ_-(\tilde \gamma)$ of a Reeb orbit $\tilde \gamma$ in a Morse--Bott family $\Sigma_{\tilde T}$  coincides with
$$
\CZ_-(\tilde \gamma) = \RS(\Sigma_{\tilde T}) -\frac{1}{2}(\dim \Sigma_{\tilde T}  -1).
$$
It follows from \eqref{eq: iterCZ-inequality} that
$$
\RS(\Sigma_T) -\frac{1}{2}(\dim \Sigma_T  -1) > \RS(\Sigma_{T_P}) -\frac{1}{2}(\dim \p W -1) \geq n+1
$$
where the latter inequality is due to Corollary \ref{cor: RSlowerbound}. This proves the claim.

Recall from Remark \ref{rem: eqdefcSH} that the capacity $c^{\SH}(W)$ is given by
$$
c^{\SH}(W) = \inf \{a >  0 \;|\; \exists \; c \in \SH_{n+1}^{+, a}(W) \text{ such that $\iota^{a}(c) = [W, \p W]$} \}.
$$
In view of the full Morse--Bott formalism in \cite[Section 3.1]{Ue16} (see also \cite[Lemma 2.4]{Bo02}), generators of $\SH_{n+1}^{+, a}(W)$ correspond to Reeb orbits in Morse--Bott families $\Sigma_T$ with
$$
\RS(\Sigma_T) -\frac{1}{2}(\dim \Sigma_T  -1) \leq n+1 \leq  \RS(\Sigma_T) +\frac{1}{2}(\dim \Sigma_T  -1).
$$
It follows from \eqref{eq: minCZ>n+1} that Reeb orbits of period $> T_P$ do not contribute to the infimum of the definition of $c^{\SH}(W)$. Consequently $c^{\SH}(W) \leq T_P$.
\end{proof}

A useful property of filtered symplectic homology is the ``uniform instability'' given in \cite[Proposition 3.5]{GinzburgShon}. The following is a version adapted to our current setup. Below, $\epsilon$ is a small positive number from \eqref{eq: SHandHo}. 

\begin{proposition}\label{prop: unif_instability}
Let $(W, \lda)$ be a Liouville domain with vanishing symplectic homology. Suppose that the induced contact from $\alpha$ on the boundary yields a periodic Reeb flow of period $T_P$. Then, for any $c \geq c^{\SH}(W)$, the natural inclusion map 
$$
\SH^{a + \epsilon}(W) \rightarrow \SH^{a + c + \epsilon}(W) 
$$
is vanishing for $a = m T_P$ with $m \in \Z_{\geq 0}$. 
\end{proposition}

\begin{proof}
The proof is essentially contained in \cite[Proposition 3.5]{GinzburgShon}. The main ingredient is the unital ring structure via the pairs-of-pants product on the symplectic homology $\SH(W)$. We refer the reader e.g. to \cite[Section 6]{Rit} for details on the ring structure. 

Consider the following diagram
\begin{equation}\label{eq: diagunifinst}
\begin{tikzcd} 
\SH^{a + \epsilon/2}(W) \otimes \SH^{\epsilon/2}(W) \arrow[r]  \arrow[d, "\iota \otimes j"]&  \SH^{a + \epsilon}(W) \arrow[d] \\
\SH^{a + c+ \epsilon/2}(W) \otimes \SH^{c + \epsilon/2}(W) \arrow[r] & \SH^{a + c+ \epsilon}(W)
\end{tikzcd}
\end{equation}
where the vertical maps are the natural inclusions and the horizontal maps are from the pairs-of-pants product. Since $c + \epsilon/2 > c^{\SH}(W)$, we have that $j^{c + \epsilon/2}[W, \p W] = 0$ by the definition of $c^{\SH}(W)$. The unit element $1_W\in \SH^{\epsilon/2}(W)$ corresponds to the fundamental class $[W, \p W]$ via the identification \eqref{eq: SHandHo}. Therefore, under the inclusion $j: \SH^{\epsilon/2}(W) \rightarrow \SH^{c + \epsilon/2}(W)$, the image of the unit $1_W$ is zero. 

Take an element $\sigma \in \SH^{a + \epsilon}(W)$. Since $a = mT_P$ and $\epsilon$ is small, it follows that 
$$
\SH^{(a + \epsilon/2, a + \epsilon)}(W) = 0
$$ 
as there are no generators in the action window $(a + \epsilon/2, a + \epsilon)$. In view of the tautological exact sequence \eqref{eq: tau_LES}, this implies that the natural inclusion $\SH^{a + \epsilon/2}(W) \rightarrow \SH^{a + \epsilon}(W)$ is an isomorphism. Let us denote the element in $\SH^{a + \epsilon/2}(W)$ which corresponds to $\sigma \in \SH^{a + \epsilon}(W)$ still by $\sigma$. Now, with respect to the diagram \eqref{eq: diagunifinst}, we see that
$$
\begin{tikzcd} 
\sigma \otimes 1_W \arrow[mapsto, r]  \arrow[mapsto, d]&  \sigma \arrow[mapsto, d] \\
\iota(\sigma) \otimes 0 \arrow[mapsto, r] & 0
\end{tikzcd}
$$
and this completes the proof.
\end{proof}

The last essential ingredient for Theorem \ref{thm: result4} is a certain periodicity in filtered symplectic homology for periodic Reeb flow.  A Morse--Bott contact form $\alpha$ is called \emph{index-positive} \cite[Definition 3.10]{Ue19} if $\RS(\Sigma_T) - \frac{1}{2}(\dim \Sigma_T -1) > 3-n$ for every Morse--Bott family $\Sigma_T$; see also \cite[Definition 3.6]{Laz20}. It is straightforward that dynamically convex Morse--Bott contact forms are index-positive. The following fact is extracted from the isomorphism given in \cite[Equation (18)]{Ue19} which can be constructed when the Reeb flow on the contact boundary is periodic and index-positive.

\begin{proposition}\label{lem: isomviaSeidelmap}
Let $(W, \lda)$ be a Liouville domain with vanishing symplectic homology and $c_1(TW) = 0$. Suppose that the induced contact from $\alpha$ on the boundary yields a periodic Reeb flow of period $T_P$ and is index-positive. Then we have a canonical isomorphism
\begin{equation}\label{eq: filteredseidelmap}
\SH^{(a, b)}(W; \Z_2) \cong \SH^{(a + T_P, b + T_P)}(W;\Z_2)
\end{equation}
for $0 < a < b \in \R$. 
\end{proposition}
For the construction of the isomorphism, we refer the reader to \cite[Section 3.4]{Ue19}. We now give a proof of Theorem \ref{thm: result4}:

\begin{proof}[Proof of Theorem \ref{thm: result4}]

In this proof, (symplectic) homology groups are over $\Z_2$-coefficients. By Lemma \ref{lem: cSHforperflow} and Proposition \ref{prop: unif_instability}, the natural inclusion map
$$
\SH^{a + \epsilon}(W) \rightarrow \SH^{a + T_P + \epsilon}(W)
$$
vanishes for any $a = m T_P$ with $m \in \Z_{\geq 0}$. From the tautological exact sequence \eqref{eq: tau_LES}, which is split exact over fields coefficients, it follows that
\begin{equation}\label{eq: directsumfSH}
\SH^{(\epsilon + mT_p, \epsilon+(m +1)T_P)}(W) \cong \SH^{\epsilon + mT_P}(W) \oplus \SH^{\epsilon + (m+1)T_P}(W).
\end{equation}

Proposition \ref{lem: isomviaSeidelmap} tells us that
$$
\rk \SH^{\epsilon + m T_P}(W) = \rk \SH^{\epsilon + (m+2)T_P}(W)
$$
for each $m \geq 0$. In particular, for $m \geq 0$ even, via the identification \eqref{eq: pSHandHo} we have
$$
\rk \SH^{\epsilon + m T_P}(W) = \rk \SH^{\epsilon}(W) = \rk \Ho(W, \p W).
$$
This proves the first assertion of the theorem.

Now assume that the contact boundary $(\p W, \alpha)$ is in addition Boothby--Wang. In this case we can completely determine the filtered symplectic homology group $\SH^a(W)$ using a standard Morse--Bott technique. First, note that the spectrum of the Reeb flow on $(\p W, \alpha)$ is given by
$$
\Spec(\p W, \alpha) = \{m T_P \;|\; m \in \Z_{\geq 1}\}.
$$
Morse--Bott families are set-wise the same as
$$
\Sigma_{m T_P} = \p W,
$$
and its Robbin--Salamon index satisfies
$$
\RS(\Sigma_{m T_P}) = m \RS(\Sigma_{T_P}).
$$
See Remark \ref{rem: meamindexandRSindex}. As in \cite[Theorem 5.4]{KvK}, there is a homological spectral sequence $\{E^r_{pq}, d^r_{pq}\}$ with 
$$
d^{r}_{pq}: E^r_{pq} \rightarrow E^r_{p - r, q + r -1}
$$ 
converging to $\SH(W)$, whose first page is given by
\begin{equation}\label{eq: MBss}
E^1_{pq} = \begin{cases} \Ho_{p+q - p\RS(\Sigma_{T_P}) + n-1}(\p W) & p \geq 1,  \\ \Ho_{q + n }(W, \p W) & p =0, \\ 0 & \text{otherwise.} \end{cases}
\end{equation}
We refer the reader to \cite[Appendix B]{KvK} for a detailed construction of the spectral sequence using the action filtration. 

Note that $E^1_{0n} = \Ho_{2n}(W, \p W) = \Z_2$. Since $\SH(W) = 0$ by the assumption, there must be a nontrivial generator at $E^1_{pq}$ with total degree $p+q = n+1$ to kill the generator at $E^1_{0n}$. On the other hand, by Corollary \ref{cor: RSlowerbound}, for each $p \geq 1$,
$$
p \RS(\Sigma_{T_P})  \geq 2np.
$$
From this we can deduce from \eqref{eq: MBss} that the total degree $p+q$ of nontrivial generators at $E^1_{pq}$ with $p \geq 1$ is at least $n+1$ and moreover attains the minimum when $p = 1$. Indeed, note that the `bottom' generator on each column with $p \geq 1$ has total degree satisfying
$$
p+ q  = p \RS(\Sigma_{T_P})-n +1 \geq 2np-n+1 \geq n+1.
$$
Here the equality holds if and only if $p =1$. Therefore we must have that $E^1_{1n} = \Ho_0(\p W) = \Z_2$ and $\RS(\Sigma_{T_P}) = 2n$. See Figure \ref{fig: E1example} on the left, describing potential positions of nontrivial generators when $n=3$ in the current situation. Combining the assumption that $\SH(W) = 0$, we obtain following consequences just for degree reasons.
\begin{itemize}
\item $\Ho_*(W, \p W) \cong \Ho_*(B^{2n}, \p B^{2n})$.
\item $\Ho_*(\p W) \cong \Ho_*(S^{2n-1})$.
\item The differential $d^1_{pq}$ is trivial except when $q = (2p - 1)n+1-p$ for $p \geq 1$.
\end{itemize}
The $E^1$-page with $n=3$ is now as in Figure \ref{fig: E1example} on the right. In particular, the spectral sequence terminates at the $E^2$-page, and the $E^1$-page for $\SH(W)$ is identical to the one for $\SH(B^{2n})$ including the differential $d^1$. 
Since the spectral sequence is constructed via the action filtration, it follows that, for each integer $m \geq 0$, the generators on the $E^2$-page with $p \leq m$ exactly contribute to the filtered symplectic homology $\SH^{\epsilon + m T_P}(W)$ where $\epsilon>0$ is sufficiently small. In other words,
$$
\SH_k^{\epsilon + m T_P}(W) \cong \bigoplus_{\substack{p+q = k \\ p \leq m}} E^2_{pq}(W) \cong  \bigoplus_{\substack{p+q = k \\ p \leq m}} E^2_{pq}(B^{2n}) \cong  \SH_k^{\epsilon + m T_P}(B^{2n}).
$$
From the above computation of the $E^2$-page, we conclude that
$$
 \SH^{\epsilon + m T_P}(W) \cong \Ho(W, \p W) \cong \Ho(B^{2n}, \p B^{2n}).
$$
\end{proof}

\begin{example}\label{ex: MBssball}
In the situation of Example \ref{ex: ballexampleDC}, we can  explicitly compute the $E^1$-page in \eqref{eq: MBss}. The case when $n=3$ is given in Figure \ref{fig: E1example} on the right; a dot denotes a nontrivial generator. Here the nontrivial differential maps are only at $(p, q)$ with $p + q = 6m + 4$ for each integer $m \geq 0$. Observe that the spectral sequence therefore terminates at the $E^2$-page.

\begin{figure}[htp]
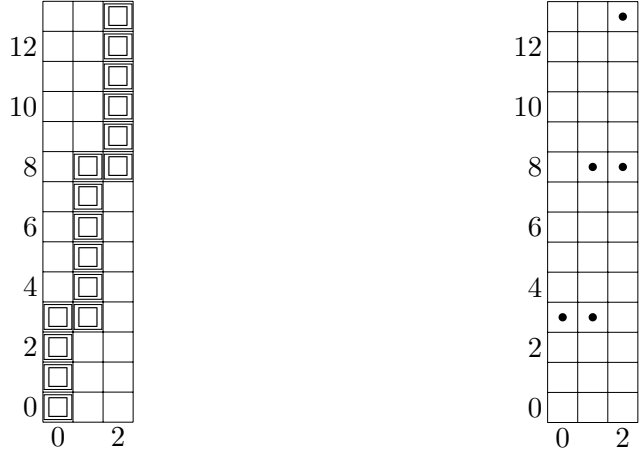

\begin{subfigure}[b]{0.25\textwidth}
\centering
\begin{sseq}{{0}...2}
{{0}...{13}}

\ssmoveto 0 0
\ssdropboxed{\square}
\ssmove 0 1
\ssdropboxed{\square}
\ssmove 0 1
\ssdropboxed{\square}
\ssmove 0 1
\ssdropboxed{\square}

\ssmoveto 1 3
\ssdropboxed{\square}
\ssmove 0 1
\ssdropboxed{\square}
\ssmove 0 1
\ssdropboxed{\square}
\ssmove 0 1
\ssdropboxed{\square}
\ssmove 0 1
\ssdropboxed{\square}
\ssmove 0 1
\ssdropboxed{\square}


\ssmoveto {2} {8}
\ssdropboxed{\square}
\ssmove 0 1
\ssdropboxed{\square}
\ssmove 0 1
\ssdropboxed{\square}
\ssmove 0 1
\ssdropboxed{\square}
\ssmove 0 1
\ssdropboxed{\square}
\ssmove 0 1
\ssdropboxed{\square}


\end{sseq}
\end{subfigure}\hspace{7.0em}
\begin{subfigure}[b]{0.25\textwidth}
\centering
\begin{sseq}{{0}...2}
{{0}...{13}}
\ssmoveto 0 {3}
\ssdropbull

\ssmoveto 1 3
\ssdropbull
\ssmove 0 5
\ssdropbull


\ssmoveto {2} {8}
\ssdropbull
\ssmove 0 5
\ssdropbull


\end{sseq}

\end{subfigure}
\caption{The $E^1$-pages}
\label{fig: E1example}
\end{figure}

\end{example}

\begin{remark}\label{rem: meamindexandRSindex}
For a periodic Reeb flow of common period $T_P$, the Robbin--Salamon index $\RS(\Sigma_{T_P})$ agrees with the \emph{mean index} $\Delta(\Sigma_{T_P})$; we refer the reader to \cite[Section 4.1]{GiGu20}, \cite[Section 5.2]{KvK}, and \cite{SaZe92}  for the precise definition of the mean index and the relation with the Robbin--Salamon index. The mean index is homogeneous under iteration, i.e. $\Delta(\Sigma_{mT_P}) = m \Delta(\Sigma_{T_P})$ for each integer $m \geq 1$  \cite[Section 4.2]{GiGu20}. It follows that
$$
\RS(\Sigma_{mT_P}) = m \RS(\Sigma_{T_P}).
$$
\end{remark}

\subsection{Relation to a conjecture of Bowden--Crowley} \label{sec: conjBC} It is conjectured by Bowden--Crowley \cite[Conjecture 2.6]{BoCr21} that \emph{any compactly supported symplectomorphism of a flexible Weinstein structure is isotopic/homotopic to the identity}. Theorem \ref{thm: result4} can be seen as a clue in the positive
 for the conjecture in terms of symplectic homology. 

Let $(W, \lda)$ be a Liouville domain with vanishing symplectic homology and $c_1(TW) = 0$. Suppose that the natural contact form $\alpha = \lda|_{\p W}$ on the boundary yields a Reeb flow of period $T_P$ as in Theorem \ref{thm: result4}. The periodic Reeb flow defines a notable class of compactly supported symplectomorphisms $\tau$ on the completion $\widehat W$, called \emph{fibered twists}: Consider a Hamiltonian $H: \widehat W \rightarrow \R$ such that $H \equiv 0$ in the interior $W$ and $H$ has slope $T_P$ at the end. We define $\tau = \tau_H$ to be the time $1$-map of the Hamiltonian diffeomorphism generated by $H$. The symplectic isotopy class of $\tau$ does not depend on the choice of $H$. For more details, see e.g. \cite[Section 4]{KKL18}.

Fibered twists $\tau$ have provided important examples of compactly supported symplectomorphisms which are smoothly isotopic to $\id$ but \emph{not} symplectically \cite{KKL18, Se00, Ul17}. It might therefore be worth paying our special attention to them for examining the conjecture.
In particular, the result in \cite[Theorem 1.1]{Ul17} gives a necessary condition for $\tau$ to be symplectically isotopic to $\id$ using a variant of symplectic homology. In our current setup, this can be interpreted as follows. (Recall that $W$ is a Liouville domain with $\SH(W) = 0$ and $c_1(TW) = 0$ whose contact boundary admits a periodic Reeb flow.)
\begin{proposition}\label{prop: perUlja}
If a fibered twist $\tau$ is symplectically isotopic to $\id$, then 
\begin{equation}\label{eq: perUlja}
\SH^{\epsilon + m T_P}(W) \cong \Ho(W, \p W)
\end{equation}
for each integer $m \geq 0$. 
\end{proposition}
\begin{proof}
By \cite[Theorem 1.1]{Ul17}, the Floer homology group for the identity map satisfies 
$$
\HF(\id, \epsilon) \cong \HF(\id, \epsilon + m T_p).
$$
We refer the reader to \cite[Section 2]{Ul17} for a definition of the above groups. By definition, the group $\HF(\id, a)$ can be identified to the filtered symplectic homology $\SH^a(W)$ \cite[Remark 2.26]{Ul17}. It follows from \eqref{eq: SHandHo} that
$$
\Ho (W, \p W) \cong \SH^{\epsilon}(W) \cong \SH^{\epsilon + m T_p}(W).
$$
This completes the proof.
\end{proof} 

\begin{example}
Fibered twists $\tau$ on $\C^n = \widehat{B^{2n}}$ associated to the periodic Reeb flow in \eqref{eq: perflowonsphere} on the standard contact sphere $(S^{2n-1}, \alpha_{\st})$ are symplectically isotopic to $\id$; it can be seen as the monodromy of a trivial Lefschetz fibration on $\C^{n+1}$. An explicit computation of filtered symplectic homology, using Morse--Bott techniques as in Example \ref{ex: MBssball}, shows that $\SH^{\epsilon + 2\pi m}(B^{2n}) \cong \Z_2 \cong \Ho(B^{2n}, \p B^{2n})$ for each integer $m \geq 0$. This is an example of Proposition \ref{prop: perUlja}.
\end{example}

It is well-known that $\SH(W)= 0$ when $W$ is flexible \cite{BoEkEl12}. Therefore Theorem \ref{thm: result4} partially establishes the periodic pattern \eqref{eq: perUlja} in symplectic homology under the dynamical convexity of the contact boundary. In the spirit of the conjecture, one may further expect that the periodic pattern in \eqref{eq: perUlja} holds without the dynamical convexity assumption, which can be seen as a version of the conjecture of Bowden--Crowley for fibered twists in terms of filtered symplectic homology:

\begin{question}
Let $(W, \lda)$ be a flexible domain, and suppose that the induced contact form $\alpha = \lda|_{\p W}$ on the boundary yields a periodic Reeb flow of period $T_P$. Then, does it hold true that
$$
\SH^{\epsilon + m T_P}(W) \cong \Ho(W, \p W)
$$
for each integer $m \geq 0$?
\end{question}

\section{Symplectic capacities of periodic dynamically convex Liouville domains}\label{sec: ThmC}

\subsection{Morse--Bott Reeb orbit property}  In this section, singular homology and symplectic homology groups are over $\Q$. The following is a generalization of the Reeb orbit property for nondegenerate Liouville domains in Proposition \ref{prop: axiomaticproportyofc_i} to the case of Morse--Bott Liouville domains.

\begin{proposition}[Morse--Bott Reeb orbit property]\label{prop: MBReebprop}
Let $(W, \lda)$ be a Morse--Bott Liouville domain. If $c_k^{\GH}(W) < \infty$, then 
$$
c_k^{\GH}(W) = T \in \Spec(\p W, \alpha)
$$
where $\displaystyle\RS(\Sigma_T) \in \left[2k + n-1 - \frac{1}{2}(\dim \Sigma_T -1), 2k + n-1 + \frac{1}{2}(\dim \Sigma_T -1)\right]$.


\end{proposition}

To give a proof, we need a lemma for Liouville domains which are partially nondegenerate:

\begin{definition}\label{def: Tnondeg}
Let $T >0$ be a real number. A contact form $\alpha$ is called \emph{$T$-nondegenerate} if every periodic Reeb orbit of action less than $T$ is nondegenerate. 
\end{definition}

\begin{lemma}\label{lem: ReeborbitpropertyforTnondegenerate}
Let $(W, \lda)$ be a $T$-nondegenerate Liouville domain with $c_1(TW) = 0$. If the capacity $c_k^{\GH}(W) < T$, then
$$
c_k^{\GH}(W) = \mathcal{A}(\gamma)
$$
for some nondegenerate Reeb orbit $\gamma$ in $(\p W, \alpha)$ with $\mathcal{A}(\gamma) < T$ and
$$
\CZ(\gamma) = 2k + n-1.
$$
\end{lemma}

\begin{proof}
Recall from Section \ref{sec: shcapacities} that 
$$
c_k^{\GH}(W) = \displaystyle \inf_{f>1} c_k^{\GH}(W_{f})
$$ 
where $f: \p W \rightarrow \R_{ > 1}$ is a smooth function and $W_{f}$ is the corresponding nondegenerate graphical subdomain of $\widehat W$. Combining with the Reeb orbit property in Proposition \ref{prop: axiomaticproportyofc_i}, we can take a sequence of nondegenerate Reeb orbits $\gamma_i$ of the contact form $\alpha_i : = \lda|_{\p W_{f_i}}$ with $\CZ(\gamma_i) = 2k + n-1$ such that 
$$
c_k^{\GH}(W) = \lim_{i\rightarrow \infty} \mathcal{A}(\gamma_i).
$$
An Arzel\`a--Ascoli argument then yields the existence of a Reeb orbit $\gamma$ of the contact form $\alpha = \lda|_{\p W}$ as the limit of $\gamma_i$'s such that
$$
c_k^{\GH}(W) = \mathcal{A}(\gamma) < T.
$$
Note here that $\gamma$ must be nondegenerate since $\alpha$ is $T$-nondegenerate. We may assume that $\alpha_i$ is chosen to be $C^2$-convergent to $\alpha$ by the genericity result in \cite[Proposition 6.1]{HWZ}. From the property in \cite[Lemma 4.10]{Mc16} it follows that
$$
\CZ(\gamma_i) = \CZ(\gamma)
$$
for sufficiently large $i$. Consequently $\CZ(\gamma) = 2k + n-1$.
\end{proof}

We now give a proof of the Morse--Bott Reeb orbit property:

\begin{proof}[Proof of Proposition \ref{prop: MBReebprop}]
Let $D_i > 0$ be a sequence of real numbers with $D_i < D_{i+1}$ and $D_i \rightarrow \infty$ as $i \rightarrow \infty$. Then the standard Morse--Bott perturbation scheme as in \cite[Lemma 2.3]{Bo02} produces a sequence of contact forms $\alpha_i$ for the contact boundary $\p W$ such that
\begin{itemize}
\item $\alpha_i \geq \alpha_{i+1}$ for $i \in \N$ (meaning that $\alpha_{i} = f \alpha_{i+1}$ for a smooth function $f: \p W \rightarrow \R$ with $f \geq 1$);
\item $\alpha_i$ converges to $\alpha$ in $C^2$-topology;
\item $\alpha_i$  is $D_i$-nondegenerate as in Definition \ref{def: Tnondeg}.
\end{itemize}
Let $W_i \subset \widehat W$ be the graphical subdomain such that $\widehat{\lda}|_{\p W_i} = \alpha_i$ as in Section \ref{sec: shcapacities}.
We claim that
\begin{equation}\label{eq: ckWinfckWi}
c_k^{\GH}(W) = \inf_{i} c_k^{\GH}(W_i).
\end{equation}
Since $\alpha_i \geq \alpha$ for all $i \in \N$, we have $W \subset W_i$, and hence the monotonicity in Proposition \ref{prop: axiomaticproportyofc_i} yields
$$
c_k^{\GH}(W) \leq \inf_{i} c_k^{\GH}(W_i).
$$
Now let $W_{f'}$ be a graphical subdomain with respect to a nondegenerate perturbation $\alpha_{f'} : = f' \alpha$ for some smooth function $f': \p W \rightarrow \R_{ > 1}$. Recall that $c_k^{\GH}(W)$ is the infimum of $c_k(W_f)$ over such nondegenerate perturbations $f$. Since $\p W$ is compact and $f_i$ converges to $1$ in $C^2$-topology, it follows that there exists $i_0 \in \N$ such that $W_i \subset W_{f'}$ for $i \geq i_0$. Therefore we have, for $i \geq i_0$,
$$
c_k^{\GH}(W_i) \leq c_k^{\GH}(W_{f'}). 
$$
It follows that 
$$
\inf_i c_k^{\GH}(W_i) \leq \inf_{f > 1} c_k^{\GH}(W_f) = c_k^{\GH}(W).
$$
This completes the proof of the equality \eqref{eq: ckWinfckWi}.

Since the sequence $D_i$ is increasing and $D_i \rightarrow \infty$ as $i \rightarrow \infty$, the capacity $c_k^{\GH}(W_i) < D_i$ for sufficiently large $i$. Note also that $(W_i, \lda_i)$ is chosen via the Morse--Bott perturbation \cite[Lemma 2.3]{Bo02}  so that it is $D_i$-nondegenerate in the sense of Definition \ref{def: Tnondeg}. (Here $\lda_i : = \widehat{\lda}|_{W_i}$, which yields $\lda_i|_{\p W_i} = \alpha_i$.) Applying the Reeb orbit property in Lemma \ref{lem: ReeborbitpropertyforTnondegenerate}, it follows that $c_k^{\GH}(W_i) = \mathcal{A}(\gamma_i)$ for some Reeb orbit $\gamma_i$ of $(\p W, \alpha_i)$ with $
\CZ(\gamma_i)= 2k+n-1.$ In view of \eqref{eq: ckWinfckWi}, an Arzel\`a--Ascoli argument shows that the sequence $\gamma_i$ has a subsequence converging to a Reeb orbit $\gamma$ of $(\p W, \alpha)$ such that $c_k^{\GH}(W) = \mathcal{A}(\gamma)$. Note that $\alpha_i$ converges to $\alpha$ in $C^2$-topology, and in this case, by applying \cite[Lemma 4.10]{Mc16}, we know that the Robbin--Salamon index $\RS(\gamma)$ satisfies 
$$
\CZ(\gamma_i) = 2k+n-1 \in \left[\RS(\gamma) - \frac{1}{2}(\dim \Sigma_T  - 1), \RS(\gamma) + \frac{1}{2}(\dim \Sigma_T  - 1)\right]
$$
where $T = \mathcal{A}(\gamma)$. This completes the proof.
\end{proof}

\subsection{Proof of Theorem \ref{thm: result3} }

\begin{lemma}\label{lem: eqnoneqcapacitysame}
Let $(W, \lda)$ be a nondegenerate dynamically convex Liouville domain with vanishing symplectic homology. Then
$$
c^{\SH}(W) = c_1^{\GH}(W).
$$
\end{lemma}

\begin{proof}
We first prove that $c^{\SH}(W) \geq c_1^{\GH}(W)$ for Liouville domains $W$ with $\SH(W) = 0$, not necessarily dynamically convex. Consider the following commutative diagram.
\begin{equation}\label{eq: commdiag_gysin}
\begin{tikzcd} 
\SH_{n+1}^{+, a}(W) \arrow[r, "\iota^a"]  \arrow[d, "g^+"]&  \Ho_{2n}(W, \p W) \arrow[r, "j^a"] \arrow[d, "f"] & \SH_n^a(W) \arrow[d, "g"] \\
\SH_{n+1}^{S^1, + , a}(W) \arrow[r, "\iota^{S^1, a}"] & \Ho_{2n}^{S^1}(W, \p W) \arrow[r, "j^{S^1, a}"] & \SH_n^{S^1, a}(W) 
\end{tikzcd}
\end{equation}
Here, the upper horizontal maps are from the exact sequence \eqref{eq: tau_LES_Ho} and the lower horizontal maps are its equivariant counterpart, and the vertical maps are from the respective Gysin exact sequence; see Section \ref{sec: eSH}. In particular $f$ sends the fundamental class $[W, \p W]$ to the fundamental class $[W, \p W]^{S^1}$. Suppose that 
$$
j^a[W, \p W] = 0.
$$
Then by the commutativity we have
$$
j^{S^1, a}[W, \p W]^{S^1} = (j^{S^1, a} \circ f) [W, \p W] = (g \circ j^a)[W, \p W] = g(0) = 0.
$$
This implies that
$$
c^{\SH}(W) \geq c_1^{\GH}(W)
$$
as $c_1^{\GH}(W) = \inf \{a > 0 \;|\;  j^{S^1, a}[W, \p W]^{S^1} = 0\}$; see Proposition \ref{prop: eqdefc1gh}.

Conversely, we prove that $c^{\SH}(W) \leq c_1^{\GH}(W)$ assuming in addition that $W$ is dynamically convex. Consider part of the Gysin exact sequence 
$$
\SH_{n}^{S^1, +, a}(W) \rightarrow \SH_{n+1}^{+, a}(W) \xrightarrow{g^+} \SH_{n+1}^{S^1, +,  a}(W) \rightarrow \SH_{n-1}^{S^1, +,  a}(W)
$$
in positive symplectic homology. Since $W$ is dynamically convex, $\SH_{k}^{S^1, +, a}(W)$ is vanishing for $k < n+1$ by Proposition \ref{prop: vanishingeSHlowdeg}. It follows that the middle map
$$
\SH_{n+1}^{+, a}(W) \xrightarrow{g^+} \SH_{n+1}^{S^1,+, a}(W)
$$
is an isomorphism. In view of the left commutative square in the diagram \eqref{eq: commdiag_gysin}, we deduce that if there exists a class $c \in \SH_{n+1}^{S^1, + a}(W)$ with $\iota^{S^1, a}(c) = [W, \p W]^{S^1}$ for some $a > 0$, then $
f(\iota^a \circ (g^+)^{-1}(c)) = [W, \p W]^{S^1}$. Since $f$ is injective and $f([W, \p W]) = [W, \p W]^{S^1}$, we have $(\iota^a \circ (g^+)^{-1})(c) = [W, \p W]$.
%
It follows that
$$
c^{\SH} (W) \geq c_1^{\GH}(W) 
$$
as $c^{\SH}(W)$ is given by
$$
c^{\SH}(W) = \inf \{a >  0 \;|\; \exists \; c \in \SH_{n+1}^{+, a}(W) \text{ such that $\iota^{a}(c) = [W, \p W]$} \}.
$$
See Remark \ref{rem: eqdefcSH}. This completes the proof.
\end{proof}

We now give a proof of Theorem \ref{thm: result3}.

\begin{proof}[Proof of Theorem \ref{thm: result3}]
For the first assertion, suppose that $(\Sigma, \alpha)$ is Boothby--Wang of common period $T_P$. Denote $c_1^{\GH}(W) = T_1$ and $c_n^{\GH}(W) = T_n$. Notice that both $T_1$ and $T_n$ are some multiples of $T_P$ and $\Sigma_{T_1} =  \Sigma_{T_n} = \Sigma_{T_P} = \Sigma$ as the flow is Boothby--Wang.

From the property in Proposition \ref{prop: MBReebprop}, we deduce that
$$
\RS(\Sigma_{T_1}) \leq 2n \leq \RS(\Sigma_{T_n}) \leq 4n-2.
$$
On the other hand, by the dynamical convexity assumption, we know $\RS(\Sigma_{T_P}) \geq 2n$ by Corollary \ref{cor: RSlowerbound}. In view of the homogeneity of the Robbin--Salamon index of principal orbits in Remark \ref{rem: meamindexandRSindex}, we know that $\RS(\Sigma_{T_1})$ and $\RS(\Sigma_{T_n})$ are some multiples of $\RS(\Sigma_{T_P})$. It follows that $\RS(\Sigma_{T_1}) = 2n = \RS(\Sigma_{T_n})$, and hence we have
$$
c_1^{\GH}(W) = T_1 = T_P = T_n = c_n^{\GH}(W)
$$
as asserted.

For the second assertion, assume that $c_1^{\GH}(W) = c_n^{\GH}(W) = T$. We claim that $\Sigma_T = \Sigma$, which implies that the Reeb flow is periodic. From Proposition \ref{prop: MBReebprop}, we deduce that
$$
3n-1 - \frac{1}{2}(\dim \Sigma_T -1) \leq \RS(\Sigma_T) \leq n+1 + \frac{1}{2}(\dim \Sigma_T -1).
$$
It follows that $\dim \Sigma_T \geq 2n-1$, and since $\Sigma_T$ is an embedded closed submanifold of $\Sigma$ with $\dim \Sigma = 2n-1$, we have $\Sigma_T = \Sigma$ as claimed. 

We now further assume that $W$ is a nondegenerate convex domain in $\R^{2n}$ containing the origin. It is shown in \cite{AbKa22, Irie} that the symplectic homology capacity $c^{\SH}(W)$ is the contact systole i.e. the minimum value of the spectrum $\Spec(\Sigma, \alpha)$. It follows from Lemma \ref{lem: eqnoneqcapacitysame} that the first capacity $c_1^{\GH}(W) =  T_P$ takes the minimum value of $\Spec(\Sigma, \alpha)$. This implies that the Reeb flow is in fact Boothby--Wang.
\end{proof}

\subsection*{Acknowledgement} The authors would like to thank Oleg Lazarev, Otto van Koert and the anonymous referee for helpful comments. Myeonggi Kwon was supported by the National Research Foundation of Korea(NRF) grant funded by the Korea government(MSIT) (No. NRF-2021R1F1A1060118). Takahiro Oba was supported by JSPS KAKENHI Grant Numbers 20K22306, 22K13913.

\bibliographystyle{abbrv}
\bibliography{mybibfile}

\begin{thebibliography}{10}

\bibitem{AbKa22}
A.~Abbondandolo and J.~Kang.
\newblock Symplectic homology of convex domains and {C}larke's duality.
\newblock {\em Duke Math. J.}, 171(3):739--830, 2022.

\bibitem{AbMa17Ann}
M.~Abreu and L.~Macarini.
\newblock Dynamical convexity and elliptic periodic orbits for {R}eeb flows.
\newblock {\em Math. Ann.}, 369(1-2):331--386, 2017.

\bibitem{AbMa17}
M.~Abreu and L.~Macarini.
\newblock Multiplicity of periodic orbits for dynamically convex contact forms.
\newblock {\em J. Fixed Point Theory Appl.}, 19(1):175--204, 2017.

\bibitem{Bo02}
F.~Bourgeois.
\newblock {\em A Morse--Bott approach to contact homology}.
\newblock 2002.
\newblock Thesis (Ph.D.)--Stanford Univeristy.

\bibitem{Bou_Lec}
F.~Bourgeois.
\newblock Introduction to contact homology.
\newblock 2003.
\newblock
  \url{https://www.imo.universite-paris-saclay.fr/~frederic.bourgeois/papers/Berder.pdf}.

\bibitem{BoEkEl12}
F.~Bourgeois, T.~Ekholm, and Y.~Eliashberg.
\newblock Effect of {L}egendrian surgery.
\newblock {\em Geom. Topol.}, 16(1):301--389, 2012.
\newblock With an appendix by Sheel Ganatra and Maksim Maydanskiy.

\bibitem{BO_MB}
F.~Bourgeois and A.~Oancea.
\newblock Symplectic homology, autonomous {H}amiltonians, and {M}orse-{B}ott
  moduli spaces.
\newblock {\em Duke Math. J.}, 146(1):71--174, 2009.

\bibitem{BO_Gysin}
F.~Bourgeois and A.~Oancea.
\newblock The {G}ysin exact sequence for {$S^1$}-equivariant symplectic
  homology.
\newblock {\em J. Topol. Anal.}, 5(4):361--407, 2013.

\bibitem{BoCr21}
J.~Bowden and D.~Crowley.
\newblock Contact open books with flexible pages.
\newblock {\em Bull. Lond. Math. Soc.}, 55(3):1302--1313, 2023.

\bibitem{ChEd22}
J.~Chaidez and O.~Edtmair.
\newblock 3{D} convex contact forms and the {R}uelle invariant.
\newblock {\em Invent. Math.}, 229(1):243--301, 2022.

\bibitem{CiEl12}
K.~Cieliebak and Y.~Eliashberg.
\newblock {\em From {S}tein to {W}einstein and back}, volume~59 of {\em
  American Mathematical Society Colloquium Publications}.
\newblock American Mathematical Society, Providence, RI, 2012.
\newblock Symplectic geometry of affine complex manifolds.

\bibitem{CiOa18}
K.~Cieliebak and A.~Oancea.
\newblock Symplectic homology and the {E}ilenberg-{S}teenrod axioms.
\newblock {\em Algebr. Geom. Topol.}, 18(4):1953--2130, 2018.
\newblock Appendix written jointly with Peter Albers.

\bibitem{AIM_Nie}
S.~Courte.
\newblock {AIM} workshop - {C}ontact topology in higher dimensions - ({M}ay
  21-25, 2012): {Q}uestions and open problems.
\newblock 2012.
\newblock
  \url{https://aimath.org/WWN/contacttop/notes_contactworkshop2012.pdf}.

\bibitem{DaGuZh22}
J.~Dardennes, J.~Gutt, and J.~Zhang.
\newblock Symplectic non-convexity of toric domains.
\newblock {\em Commun. Contemp. Math.}, Online Ready, 2023.

\bibitem{EHcapacity}
I.~Ekeland and H.~Hofer.
\newblock Symplectic topology and {H}amiltonian dynamics.
\newblock {\em Math. Z.}, 200(3):355--378, 1989.

\bibitem{GiGu20}
V.~L. Ginzburg and B.~Z. G\"{u}rel.
\newblock Lusternik-{S}chnirelmann theory and closed {R}eeb orbits.
\newblock {\em Math. Z.}, 295(1-2):515--582, 2020.

\bibitem{GiGuMa21}
V.~L. Ginzburg, B.~Z. G\"{u}rel, and M.~Mazzucchelli.
\newblock On the spectral characterization of {B}esse and {Z}oll {R}eeb flows.
\newblock {\em Ann. Inst. H. Poincar\'{e} C Anal. Non Lin\'{e}aire},
  38(3):549--576, 2021.

\bibitem{GiMa21}
V.~L. Ginzburg and L.~Macarini.
\newblock Dynamical convexity and closed orbits on symmetric spheres.
\newblock {\em Duke Math. J.}, 170(6):1201--1250, 2021.

\bibitem{GinzburgShon}
V.~L. Ginzburg and J.~Shon.
\newblock On the filtered symplectic homology of prequantization bundles.
\newblock {\em Internat. J. Math.}, 29(11):1850071, 35, 2018.

\bibitem{Gutt}
J.~Gutt.
\newblock The positive equivariant symplectic homology as an invariant for some
  contact manifolds.
\newblock {\em J. Symplectic Geom.}, 15(4):1019--1069, 2017.

\bibitem{GH18}
J.~Gutt and M.~Hutchings.
\newblock Symplectic capacities from positive {$S^1$}-equivariant symplectic
  homology.
\newblock {\em Algebr. Geom. Topol.}, 18(6):3537--3600, 2018.

\bibitem{Hat02}
A.~Hatcher.
\newblock {\em Algebraic topology}.
\newblock Cambridge University Press, Cambridge, 2002.

\bibitem{HWZ}
H.~Hofer, K.~Wysocki, and E.~Zehnder.
\newblock The dynamics on three-dimensional strictly convex energy surfaces.
\newblock {\em Ann. of Math.}, 148(1):197--289, 1998.

\bibitem{Irie}
K.~Irie.
\newblock Symplectic homology of fiberwise convex sets and homology of loop
  spaces.
\newblock {\em J. Symplectic Geom.}, 20(2):417--470, 2022.

\bibitem{KiKiKw22}
J.~Kim, S.~Kim, and M.~Kwon.
\newblock Remarks on the systoles of symmetric convex hypersurfaces and
  symplectic capacities.
\newblock {\em J. Fixed Point Theory Appl.}, 24(2):Paper No. 28, 26, 2022.

\bibitem{KKL18}
J.~Kim, M.~Kwon, and J.~Lee.
\newblock Volume growth in the component of fibered twists.
\newblock {\em Commun. Contemp. Math.}, 20(8):1850014, 43, 2018.

\bibitem{KvK}
M.~Kwon and O.~van Koert.
\newblock Brieskorn manifolds in contact topology.
\newblock {\em Bull. Lond. Math. Soc.}, 48(2):173--241, 2016.

\bibitem{Laz20}
O.~Lazarev.
\newblock Contact manifolds with flexible fillings.
\newblock {\em Geom. Funct. Anal.}, 30(1):188--254, 2020.

\bibitem{booklong}
Y.~Long.
\newblock {\em Index theory for symplectic paths with applications}.
\newblock Progress in Mathematics. Birkh\"auser Basel, 1 edition, 2002.

\bibitem{MaRa20}
M.~Mazzucchelli and M.~Radeschi.
\newblock On the structure of {B}esse convex contact spheres.
\newblock {\em Trans. Amer. Math. Soc.}, 376(3):2125--2153, 2023.

\bibitem{McD91}
D.~McDuff.
\newblock Symplectic manifolds with contact type boundaries.
\newblock {\em Invent. Math.}, 103(3):651--671, 1991.

\bibitem{Mc16}
M.~McLean.
\newblock Reeb orbits and the minimal discrepancy of an isolated singularity.
\newblock {\em Invent. Math.}, 204(2):505--594, 2016.

\bibitem{Oa04}
A.~Oancea.
\newblock A survey of {F}loer homology for manifolds with contact type boundary
  or symplectic homology.
\newblock In {\em Symplectic geometry and {F}loer homology. {A} survey of the
  {F}loer homology for manifolds with contact type boundary or symplectic
  homology}, volume~7 of {\em Ensaios Mat.}, pages 51--91. Soc. Brasil. Mat.,
  Rio de Janeiro, 2004.

\bibitem{Rit}
A.~F. Ritter.
\newblock Topological quantum field theory structure on symplectic cohomology.
\newblock {\em J. Topol.}, 6(2):391--489, 2013.

\bibitem{RS}
J.~Robbin and D.~Salamon.
\newblock The maslov index for paths.
\newblock {\em Topology}, 32(4):827--844, 1993.

\bibitem{SaZe92}
D.~Salamon and E.~Zehnder.
\newblock Morse theory for periodic solutions of {H}amiltonian systems and the
  {M}aslov index.
\newblock {\em Comm. Pure Appl. Math.}, 45(10):1303--1360, 1992.

\bibitem{Se00}
P.~Seidel.
\newblock Graded {L}agrangian submanifolds.
\newblock {\em Bull. Soc. Math. France}, 128(1):103--149, 2000.

\bibitem{Sei08}
P.~Seidel.
\newblock A biased view of symplectic cohomology.
\newblock In {\em Current developments in mathematics, 2006}, pages 211--253.
  Int. Press, Somerville, MA, 2008.

\bibitem{Ue16}
P.~Uebele.
\newblock Symplectic homology of some {B}rieskorn manifolds.
\newblock {\em Math. Z.}, 283(1-2):243--274, 2016.

\bibitem{Ue19}
P.~Uebele.
\newblock Periodic {R}eeb flows and products in symplectic homology.
\newblock {\em J. Symplectic Geom.}, 17(4):1201--1250, 2019.

\bibitem{Ul17}
I.~Uljarevic.
\newblock Floer homology of automorphisms of {L}iouville domains.
\newblock {\em J. Symplectic Geom.}, 15(3):861--903, 2017.

\bibitem{Vit}
C.~Viterbo.
\newblock Functors and computations in {F}loer homology with applications. {I}.
\newblock {\em Geom. Funct. Anal.}, 9(5):985--1033, 1999.

\bibitem{Zh20}
Z.~Zhou.
\newblock Vanishing of symplectic homology and obstruction to flexible
  fillability.
\newblock {\em Int. Math. Res. Not. IMRN}, (23):9717--9729, 2020.

\end{thebibliography}

\end{document}